\documentclass{amsart}
\usepackage{amsmath,amssymb}
\usepackage{color,mathdots}
\usepackage{comment}

\newtheorem{theorem}{Theorem}[section]
\newtheorem{lemma}[theorem]{Lemma}
\newtheorem{corollary}[theorem]{Corollary}
\newtheorem{fact}[theorem]{Fact}

\theoremstyle{definition}
\newtheorem{example}[theorem]{Example}

\newtheorem{remark}[theorem]{Remark}
\newtheorem{definition}[theorem]{Definition}
\newtheorem{assumption}[theorem]{Assumption}

\def \d {\delta}
\def \L {\mathcal L}
\def \D {\mathcal D}

\def \UU {\mathcal U}

\def \dcl{\operatorname{dcl}}
\def \acl{\operatorname{acl}}
\def \alg{\operatorname{alg}}
\def \tp{\operatorname{tp}}
\def \qftp{\operatorname{qftp}}

\def \diag{\operatorname{diag}}
\def \Aut{\operatorname{Aut}}

\newcommand{\DF}{\textrm{\normalfont DF}}
\newcommand{\SCF}{\textrm{\normalfont SCF}}
\newcommand{\ACF}{\textrm{\normalfont ACF}}
\newcommand{\DPF}{\textrm{\normalfont DPF}}
\newcommand{\CCM}{\textrm{\normalfont CCM}}
\newcommand{\DCCM}{\textrm{\normalfont DCCM}}
\newcommand{\DCF}{\textrm{\normalfont DCF}}
\newcommand{\SCH}{\textrm{\normalfont SCH}}
\newcommand{\SDCF}{\textrm{\normalfont SDCF}}
\newcommand{\CODF}{\textrm{\normalfont CODF}}

\def\Ind#1#2{#1\setbox0=\hbox{$#1x$}\kern\wd0\hbox to 0pt{\hss$#1\mid$\hss}
\lower.9\ht0\hbox to 0pt{\hss$#1\smile$\hss}\kern\wd0}
\def\ind{\mathop{\mathpalette\Ind{}}}
\def\Notind#1#2{#1\setbox0=\hbox{$#1x$}\kern\wd0\hbox to 0pt{\mathchardef
\nn=12854\hss$#1\nn$\kern1.4\wd0\hss}\hbox to
0pt{\hss$#1\mid$\hss}\lower.9\ht0 \hbox to
0pt{\hss$#1\smile$\hss}\kern\wd0}

\title[]{Neostability transfers in derivation-like theories}

\author{Omar Le\'on S\'anchez}
\address{Omar Le\'on S\'anchez, Department of Mathematics, University of Manchester, Oxford Road, Manchester, United Kingdom M13 9PL}
\email{omar.sanchez@manchester.ac.uk}
\thanks{Omar Le\'on S\'anchez was supported by EPSRC grant EP/V03619X/1}

\author{Shezad Mohamed}
\address{Shezad Mohamed, Department of Mathematics, University of Manchester, Oxford Road, Manchester, United Kingdom M13 9PL}
\email{shezad.mohamed@manchester.ac.uk}
\thanks{}

\date{\today}

\begin{document}

\begin{abstract}
Motivated by structural properties of differential field extensions, we introduce the notion of a theory $T$ being derivation-like with respect to another model complete theory $T_0$. We prove that when $T$ admits a model companion $T_+$, several model-theoretic properties transfer from $T_0$ to $T_+$. These properties include completeness, quantifier elimination, stability, simplicity, and NSOP$_1$. We also observe that, aside from the theory of differential fields, examples of derivation-like theories are plentiful. 
\end{abstract}

\maketitle

\tableofcontents

\section{Introduction}

Extending the argument of simplicity of the theory ACFA, in \cite{CP98} Chatzidakis and Pillay studied the abstract condition of adding an automorphism to a first-order theory $T_0$ and proved that if such an expanded theory has a model companion $T_0A$, then $T_0A$ is simple whenever $T_0$ is stable (this stable-to-simple transfer result has been further generalised in \cite{BloMar19}). In this paper we propose an abstract analogue of this where instead of adding an automorphism, we expand $T_0$ to a theory $T$ that satisfies certain conditions which resemble structural properties of derivations.

\smallskip

Recall that given a difference field $(K,\sigma)$, the automorphism $\sigma$ extends (not necessarily uniquely) to the separable closure $K^{\text{sep}}$. In the case of a differential field $(K,\d)$ much more is true: the derivation extends \emph{uniquely} to any separably algebraic extension. This is a crucial difference between the theories of difference fields and differential fields; for instance, it is one of the reasons why $\DCF_0$ has quantifier elimination while ACFA does not. Another structural property of differential fields (or even differential rings) is that given two differential fields $(K,\delta_1)$ and $(L,\d_2)$ with a common differential subfield $(E,\d)$, the tensor product $K\otimes_E L$ has a \emph{unique} derivation extending those on $K$ and $L$ (note that this property also holds for difference fields). 

\smallskip

We extract the above two properties of differential field extensions to an abstract setup and define (in Definition~\ref{like}) the notion of a theory $T$ being derivation-like with respect to a complete and model complete theory $T_0$ equipped with an invariant ternary relation $\ind^0$. The motivating example, of course, is that the theory of differential fields in characteristic zero, $\DF_0$, is derivation-like with respect to $\ACF_0$ equipped with the algebraic disjointness relation $\ind^{\alg}$. In \S\ref{examples}, we provide several other instances of derivation-like theories; in particular, we note that the recently developed theory DCCM of compact complex manifolds with meromorphic vector fields, introduced by Moosa in \cite{Moo23}, is derivation-like. 

\smallskip

In \S\ref{mainresults}, under the assumption that $T$ has a model companion $T_+$ (and some assumptions on $\ind^0$), we prove that several model-theoretic properties transfer from $T_0$ to $T_+$. In particular, completeness and quantifier elimination transfer, the model-theoretic $\dcl$ and $\acl$ have a natural description, and the neostability properties of stability, simplicity, NSOP$_1$ (under additional conditions), and rosiness transfer from $T_0$ to $T_+$. 

While most neostability properties have local combinatorial descriptions, for each of the four mentioned above, there is a theorem of a form similar to the Kim--Pillay theorem -- one indicating a semantic way to characterise the given property in terms of the existence of an independence relation satisfying certain conditions. These appear in \cite{Kim13}, \cite{KimPi97}, \cite{CDHJR}, and \cite{Ad09}, respectively, and will be stated in full in Theorem~\ref{neostability} below.

By inducing (from $\ind^0$) a natural independence relation
\[
A\ind^+_C B \iff \acl(AC) \ind^0_{\acl(C)} \acl(BC)
\]
on the model companion $T_+$ and using Theorem~\ref{neostability}, in \S\ref{mainresults} we are able to prove that stability and simplicity transfer from $T_0$ to $T_+$. 

\begin{theorem}\label{intro-neostability-transfers}
Suppose $\ind^0$ is nonforking independence in $T_0$ and that for every model $M \models T_+$ we have that $\dcl_0(M) \models T_0$. Suppose also that $T$ is derivation-like with respect to $(T_0, \ind^0)$.
\begin{enumerate}
    \item If $T_0$ is stable, then $T_+$ is stable and $\ind^+$ is nonforking independence.
    \item If $T_0$ is simple, then $T_+$ is simple and $\ind^+$ is nonforking independence.
\end{enumerate}
\end{theorem}


We also prove the following transfer result for NSOP$_1$.

\begin{theorem}
Suppose $\ind^0$ is Kim-independence in $T_0$, that $T_0$ has an independence relation $\ind^1$ such that $\ind^0_M \implies \ind^1_M$ for every $M \models T_+$, and that $T_0 \subseteq T_+$. Suppose also that $T$ is derivation-like with respect to $(T, \ind^1)$.

If $T_0$ is NSOP$_1$, then $T_+$ is NSOP$_1$ and $\ind^+$ is Kim-independence.
\end{theorem}

In particular, when $T_0$ is the theory of a very slim field, the relation $\ind^{\alg}$ is a natural choice for $\ind^1$. Finally, assuming that both $T_0$ and $T_+$ eliminate imaginaries, we obtain that rosiness also transfers. 

\begin{theorem}
    Suppose that $T_0$ eliminates imaginaries and is rosy with strict independence relation $\ind^0$. Suppose also that $T$ is derivation-like with respect to $(T_0, \ind^0)$ and that $T_+$ eliminates imaginaries. Then $\ind^+$ is a strict independence relation on $T_+$, and hence $T_+$ is rosy.
\end{theorem}

Our method of proof relies on a detailed study of how the individual properties of independence relations constituting Definition~\ref{indprop} transfer from $T_0$ to $T_+$. This is explicitly done in Theorems~\ref{thm-ind-properties-transfer}, \ref{thm-ind-thm-transfers}, and \ref{thm-stationarity-transfers}.

As consequences of these theorems, we obtain the following familiar results on the stability and simplicity of theories of fields with operators.

\begin{corollary}
    \begin{enumerate}
        \item $\SCF_{p,e}$ (for $e < \infty$) is stable, and forking independence coincides with algebraic disjointness in the language with constant symbols for a fixed $p$-basis and $\lambda$-functions.

        \item $\DCCM$ is stable, and forking independence coincides with that in $\CCM$.
        
        \item The model companion of a bounded PAC differential field of characteristic $0$ is simple, and forking independence coincides with algebraic disjointness.

        \item $\CODF$ is rosy.
    \end{enumerate}
\end{corollary}

And we obtain novel results.

\begin{corollary}
\begin{enumerate}
    \item $\SDCF_{p,\infty}^\lambda$, the theory of separably differentially closed fields of characteristic $p$ and infinite differential degree of imperfection, is stable, and forking independence coincides with algebraic disjointness and $p$-disjointness in the language with the $\lambda$-functions.

    \item The model companion of an $\omega$-free PAC differential field of characteristic $0$ is NSOP$_1$.
\end{enumerate}
\end{corollary}
Stability of $\SDCF_{p,\infty}^\lambda$ was established by Ino and the first author in \cite{InoLS23} by counting types; here we give a characterisation of forking.

\medskip

\noindent {\bf Conventions.} We assume that all our theories are closed under deductions.

\medskip

\noindent {\bf Acknowledgements.} The authors would like to thank Jan Dobrowolski and Amador Martin-Pizarro for several useful discussions that led to improvements of this paper and the anonymous referee of this paper for their very helpful comments that improved the clarity of our exposition and suggestion of a new class of examples.

\

\section{Preliminaries}

As discussed in the introduction, our definitions and proofs will rely on the notion of abstract independence relations and the fact that stability, simplicity, NSOP$_1$, and rosiness all have a theorem of the form of the Kim--Pillay theorem: one characterising the given property in terms of the existence of a ternary relation satisfying certain properties. In this section, we will collect the necessary material and use it freely.

We follow Adler's treatment \cite{Ad09} of independence relations. We fix a complete first-order theory $T_0$.

\begin{definition}\label{indprop}
A relation $\ind^0$ on triples of small subsets of a monster model $\UU_0$ of $T_0$ is called an independence relation if it is invariant under automorphisms and satisfies the following eight properties:
\begin{enumerate}
\item \emph{normality:} $A\ind^0_C B \implies A\ind^0_C BC$;
\item \emph{monotonicity:} $A\ind^0_C BD\implies A\ind^0_C B$;
\item \emph{base monotonicity:} $A\ind^0_C BD \implies A\ind^0_{CD} B$;
\item \emph{transitivity:} $A\ind^0_C B \text{ and } A\ind^0_B D \implies A\ind^0_C D \text{ for } C\subseteq B\subseteq D$;
\item \emph{symmetry:} $A\ind^0_C B \implies B\ind^0_C A$;
\item \emph{full existence:} for any $A,B,C$ there is $A'\equiv_C A$ with $A'\ind^0_C B$;
\item \emph{finite character:} if $A_0\ind^0_C B$ for all finite $A_0\subseteq A$ then $A\ind^0_C B$;
\item \emph{local character:} for any $A$ there is a cardinal $\kappa=\kappa(A)$ such that for any $B$ there is $C\subseteq B$ with $|C|< \kappa$ such that $A\ind^0_{C}B$.
\end{enumerate}
There are other properties that are generally of interest:
\begin{itemize}
\item \emph{existence:} for any $A$ and $C$ we have $A\ind^0_C C$;
\item \emph{extension:} if $A\ind^0_C B$ then for any $D$ there is $A'\equiv_{BC} A$ with $A'\ind^0_C BD$;
\item \emph{anti-reflexivity:}  if $a\ind^0_C a$ then $a\in \acl(C)$ (an independence relation is called strict if it satisfies anti-reflexivity);
\item \emph{chain local character:} for a finite tuple $a$ and a regular cardinal $\kappa>|T_0|$, for every continuous chain of models $(M_i)_{i<\kappa}$ with $|M_i|<\kappa$ there is $j<\kappa$ such that $a\ind^0_{M_j}\bigcup_{i<\kappa}M_i$;
\item \emph{independence theorem over $M$:} if $A_1\ind^0_M A_2$, $a_1\ind^0_M A_1$, $a_2\ind^0_M A_2$, and $a_1\equiv_M a_2$, then there is $a\models \tp(a_1/MA_1)\cup \tp(a_2/MA_2)$ with $a\ind^0_M A_1A_2$;
\item \emph{stationarity over $M$:} if $M\subseteq A$, $a\ind^0_M A$, $b\ind^0_M A$, and $a\equiv_M b$, then $a\equiv_A b$.
\end{itemize}

We say that \emph{monotonicity}, \emph{symmetry}, \emph{finite character}, \emph{existence}, or \emph{extension} holds over models if the property holds when $C$ is a small model of $T_0$, and we say that $\ind^0_M$ satisfies the property if it holds when $C = M$. The \emph{independence theorem} or \emph{stationarity} holds over models if the property holds when $M$ is a small model of $T_0$. We say that \emph{transitivity} holds over models if the property holds when $C$ and $B$ are small models of $T_0$.
\end{definition}

We collect the various theorems of the form of the Kim--Pillay theorem.

\begin{theorem}\label{neostability} \
\begin{enumerate}
\item [(i)] \cite{Kim13} The theory $T_0$ is stable if and only it admits an independence relation $\ind^0$ (that is, one satisfying (1)-(8) above) which satisfies stationarity over models. In this case $\ind^0$ coincides with forking independence.
\item [(ii)] \cite{KimPi97} The theory $T_0$ is simple if and only it admits an independence relation $\ind^0$ which satisfies the independence theorem over models. In this case $\ind^0$ coincides with forking independence.
\item [(iii)] \cite{CDHJR} The theory $T_0$ is NSOP$_1$ if and only if it admits an invariant ternary relation $\ind^0$ with chain local character and which over models satisfies monotonicity, transitivity, symmetry, finite character, existence, extension, and the independence theorem. In this case $\ind^0$ coincides with Kim-independence over models. 
\item [(iv)] \cite{Ad09} Suppose $T_0$ eliminates imaginaries. Then, $T_0$ is rosy if and only if it admits a strict independence relation.
\end{enumerate}
\end{theorem}

Finally, we define some notation that we will use freely throughout the next section.

\begin{definition}
Let $\L$ be a language, let $M$ and $N$ be two $\L$-structures, and let $X \subseteq M$ be some subset.
    \begin{enumerate}
        \item $M \leq_\L N$ means that $M$ is an $\L$-substructure of $N$;
        \item $\langle X \rangle_{\L}$ is the $\L$-structure generated by $X$ (inside $M$);
        \item $\diag^M_{\L}(X)$ is the quantifier-free $\L$-diagram of $X$ inside $M$; that is, the set of all quantifier free $\L(C)$-sentences true in $M$, where $C = \{c_x \colon x \in X \}$ is a set of new constant symbols and $M$ is expanded to an $\L(C)$-structure by interpreting $c_x$ as $x$.
    \end{enumerate}
\end{definition}

\

\section{Main results} \label{mainresults}

We fix the following data:

\begin{itemize}
\item $\L_0\subseteq \L$ are two (first-order) languages, possibly multi-sorted;
\item $T_0$ is a complete and model complete $\L_0$-theory equipped with an automorphism invariant ternary relation $\ind^0$, we denote by $\UU_0$ a monster model of $T_0$ and, unless otherwise stated, $\acl_0$ refers to model-theoretic algebraic closure taken with respect to the language $\L_0$ in $\UU_0$; and
\item $T$ is an $\L$-theory such that $T_0^\forall \subseteq T$.
\end{itemize}

\begin{definition} \label{like}\
We say that $T$ is derivation-like with respect to $(T_0,\ind^0)$ if whenever $A,B,C\models T^\forall$, with $C$ a common $\L$-substructure of $A$ and $B$, are such that $A, B\leq _{\L_0}\UU_0$, $\acl_0(C) \cap A = \acl_0(C) \cap B = C$, and $A\ind^0_C B$, we have that
\begin{enumerate}
\item [(i)] there exists $M\models T$ such that $M\leq_{\L_0} \UU_0$ and $A, B\leq_\L M$, and 
\item [(ii)] for any $M$ as in (i) and any $\L_0$-structure $D$ such that  
$$\langle A,B\rangle_{\L_0}\leq_{\L_0}D\leq_{\L_0} \acl_0(A,B)\cap M,$$ 
we have that $D\leq_{\L}M$ and, moreover, this $\L$-structure on $D$ is the unique one expanding its $\L_0$-structure, making it a model of $T^\forall$, and extending the $\L$-structures of $A$ and $B$.   
\end{enumerate}
\end{definition}

Note that part (i) of the definition is, in some sense, a strong form of independent amalgamation.

\begin{remark}\label{aclA-and-useforCCM} \ 
\begin{enumerate}
    \item Suppose $T$ is derivation-like with respect to $(T_0,\ind^0$) and $M\models T$ with $M\leq_{\L_0} \UU_0$. We note that if $A\leq_\L M$ is such that 
$$A\ind^0_{A} A,$$
then $\acl_0(A)\cap M\leq_\L M$. 
Indeed, taking $B$ and $C$ equal to $A$, part (ii) of the definition yields 
$$\acl_0(A)\cap M=\acl_0(A,B)\cap M\leq_\L M.$$

More generally, whenever $D$ is an $\L_0$-structure such that $A\leq_{\L_0}D\leq_{\L_0}\acl_0(A)\cap M$,
we then have that $D\leq_{\L} M$ and this $\L$-structure on $D$ is the unique one expanding its $\L_0$-structure, making it a model of $T^\forall$, and extending the $\L$-structure of $A$. 

\item We say that $T$ is \emph{almost} derivation-like with respect to $(T_0, \ind^0)$ if in condition (ii) of Definition~\ref{like} we restrict only to $A=B=C$. Some of the preliminary results of this section will also hold for almost derivation-like theories.

\item When $T_0 \subseteq T$, we may weaken condition (ii) by restricting to only those cases where $D$ is $\dcl_0$-closed: the results of this section will continue to hold.
\end{enumerate}
\end{remark}

\begin{example}\label{example-T0-in-T+}
We highlight the distinction made above in Remark~\ref{aclA-and-useforCCM}(3) between $T_0 \subseteq T$ and $T_0 \not \subseteq T$.

Let $T_0 = \ACF_p$ and $T = \SCF_{p,e}$ with $e>0$. Then $T_0^\forall \subseteq T$ but $T_0 \not \subseteq T$ since fields of positive degree of imperfection are not algebraically closed.

If $T_0 = \ACF_0$ and $T = \DCF_0$, then $T_0 \subseteq T$ since every differentially closed field is necessarily algebraically closed.
\end{example}

\medskip

The following assumptions will be in place throughout the rest of this section.

\begin{assumption} \label{theassumption}\
\begin{enumerate}
\item [(i)]From now on $T$ is a derivation-like theory with respect to $(T_0,\ind^0)$. 
\item [(ii)] We assume that $T$ has a model companion $T_+$ and that $T\subseteq T_+$. We fix a monster model $\UU_+$ of a completion of $T_+$.
Since $T_0^\forall\subseteq T_+$, without loss of generality we may assume that $\UU_+\leq_{\L_0}\UU_0$. $\acl_+$ refers to model-theoretic algebraic closure taken in $\UU_+$ (with respect to the language $\L$).
\item [(iii)] If $T_0\not\subseteq T_+$, we further assume that $T_0$ has quantifier elimination.
\end{enumerate}
\end{assumption}

\medskip

Let $\L_0^*$ be some language expanding $\L_0$, and set $\L^* = \L \cup \L_0^*$. Let $T^*_0$ be an expansion by definitions of $T_0$ to the language $\L_0^*$ (for instance, the Morleyisation of $T_0$). Also, expand $T$ and $T_+$ to $T^*$ and $T_+^{*}$, respectively, to the language $\L^*$ using the same definitions as for $T_0^*$.

\begin{remark} \label{Morleyisation} The following can be readily checked:
\begin{enumerate}
    \item $(T_0^*)^\forall \subseteq T^*$;
    \item $\ind^0$ is naturally an invariant ternary relation on $\UU_0$ as a model of $T_0^*$;
    \item $T^*$ is derivation-like with respect to $(T_0^*, \ind^0)$;
    \item $T_+^{*}$ is the model companion of $T^*$ and $T^*\subseteq T_+^{*}$; and
    \item $\UU_0$ and $\UU_+$ remain monster models of $T_0^*$ and $T_+^{*}$, respectively, and $\UU_+ \leq_{\L_0^*} \UU_0$.
\end{enumerate}
For (1), in particular to obtain $\UU_+ \leq_{\L_0^*} \UU_0$, we need to ensure that any new symbols of $\L_0^*$ have compatible interpretations in $M \leq_{\L_0} N$ for $M \models T$ and $N \models T_0$; that is, that for any $\L_0$-formula $\phi$ defining a new symbol we have
\[
\phi(M) = \phi(N) \cap M.
\]
If $T_0 \subseteq T$, then $M \preceq N$ as $T_0$ is model complete, and this is automatic. If $T_0 \not \subseteq T$, this can be ensured by assuming $T_0$ has quantifier elimination in the language $\L_0$. This is the basis for Assumption~\ref{theassumption}(iii).
\end{remark}

\begin{lemma}\label{diagram}
Assume $\ind^0$ satisfies full existence. Let $A \leq_{\L} \UU_+$. If $T_0^* \cup \diag_{\L_0^*}^{\UU_0}(A)$ is complete, then $T^{*}_+ \cup \diag_{\L^*}^{\UU_+}(A)$ is complete. 
\end{lemma}
\begin{proof}
Let $K \models T_+^{*} \cup \diag_{\L^*}^{\UU_+}(A)$. We will show that $K \equiv_A \UU_+$ as $\L^*$-structures. First note that $K \models (T_0^*)^\forall$, and hence it $\L_0^*$-embeds in some $K' \models T_0^*$. Now by completeness of $T_0^* \cup \diag_{\L_0^*}^{\UU_0}(A)$, $K'$ $\L_0^*$-embeds inside $\UU_0$ over $A$. Let $L$ be an $\L^*$-elementary substructure of $\UU_+$ containing $A$. Use full existence to find a copy of $L$ with $L' \ind^0_A K$ and $\tp^{\UU_0}_{\L_0^*}(L'/A) = \tp^{\UU_0}_{\L_0^*}(L/A)$. This last fact means that $L$ induces an isomorphic $\L^*$-structure on $L'$. The partial $\L^*_0$-elementary map $A \to A$ from $K$ to $L'$ extends to a partial $\L^*_0$-elementary map $\acl_{\L_0^*}^{\UU_0}(A) \cap K \to \acl_{\L_0^*}^{\UU_0}(A) \cap L'$. By Remark~\ref{aclA-and-useforCCM}(1), this map must be an $\L^*$-isomorphism (note that full existence yields $A\ind^0_{A} A$). So we may assume that $A$ is relatively $\acl_{\L^*}$-closed in $K$ and $L'$.

Since $T^*$ is derivation-like, there is some $M \models T^*$ such that $M \leq_{\L_0^*} \UU_0$ and $K,L' \leq_{\L^*} M$. Since $T_+^{*}$ is the model companion of $T^*$, there is some $N \models T_+^{*}$ extending $M$ as an $\L^*$-structure. Now, $T_+^{*}$ is model complete, so $L' \preceq N \succeq K$ as $\L^*$-structures. Finally $K \equiv_A N \equiv_A L' \equiv_A L \equiv_A \UU_+$.
\end{proof}

We collect some immediate corollaries.

\begin{corollary}\label{corollary-qe-transfers}
Assume $\ind^0$ satisfies full existence. Suppose $T^*_0$ is the model companion of some inductive $\L_0^*$-theory $S$. Then
\begin{enumerate}
    \item $T_+ \cup S$ is the model companion of $T \cup S$;
    \item if $T^*_0$ is the model completion of $S$, then $T_+ \cup S$ is the model completion of $T \cup S$; and
    \item if $T^*_0$ has quantifier elimination, then $T^*_+$ has quantifier elimination.
\end{enumerate}
\end{corollary}
\begin{proof}
This is precisely the analogous result of Theorem~7.2 of \cite{Tre05}. The argument is given in more detail in Theorem~3.9 of \cite{Moha23}.
\end{proof}

As a result, we now {\bf assume} in all cases that $T_0$ has quantifier elimination -- if $T_0 \subseteq T$, then we Morleyise and use Remark~\ref{Morleyisation}; if $T_0 \not \subseteq T$, we assumed in Assumption~\ref{theassumption}(iii) that $T_0$ has quantifier elimination in $\L_0$. Hence by Corollary~\ref{corollary-qe-transfers}, $T_+$ also has quantifier elimination.

\begin{lemma}\label{acl}
Assume $\ind^0$ satisfies monotonicity, symmetry, full existence, and anti-reflexivity. Then, for any $A\subset  \UU_+$, we have
$$\acl_+(A)=\acl_0(\langle A\rangle_\L)\cap \UU_+.$$
\end{lemma}
\begin{proof}

By full existence, we have that $A\ind^0_A A$, and so by Remark~\ref{aclA-and-useforCCM}(1) we have that
$$F:=\acl_0(\langle A\rangle _\L)\cap \UU_+$$ 
is an $\L$-substructure of $\UU_+$. As $T_0$ has quantifier elimination, we get $F\subseteq \acl_+(A)$. For the other containment, consider $a\in \acl_+(A)$. Let $K$ be an elementary $\L$-substructure of $\UU_+$ containing all (finitely many) realisations of $\tp_+(a/F)$. By full existence, there is an $\L_0$-substructure $L$ of $\UU_0$ such that $L\ind^0_F K$ and $\tp_0(L/F)=\tp_0(K/F)$. The latter induces an $\L$-structure on $L$, via some $\sigma\in \Aut_{\L_0}(\UU_0/F)$ with $L=\sigma(K)$, making $L$ a model of $T_+$ and an $\L$-extension of $F$. Since $T$ is derivation-like, there is $M\models T$ with $M\leq_{\L_0}\UU_0$ such that $K$ and $L$ are $\L$-substructures of $M$. Let $N\models T_+$ be an $\L$-extension of $M$. Since $T_+$ is model complete and $K\models T_+$ is a common substructure of $N$ and $\UU_+$, there is an elementary $\L$-embedding $\phi:N\to \UU_+$ over $K$. Let $L'=\phi(L)$. We first note that 
$$\tp_+^{\UU_+}(a/F)=\tp_+^{K}(a/F)=\tp_+^L(\sigma(a)/F)=\tp_+^N(\sigma(a)/F)=\tp_+^{\UU_+}(\phi(\sigma(a))/F)$$
and so $\phi(\sigma(a))\in K$ (as $K$ contains all realisations of $\tp_{+}^{\UU_+}(a/F)$).

We now claim that $\tp_0(L/K)=\tp_0(L'/K)$. First note that
$$\qftp_0^{\UU_0}(L/K)=\qftp_0^{M}(L/K)=\qftp_0^{N}(L/K)=\qftp_0^{\UU_+}(L'/K)=\qftp_0^{\UU_0}(L'/K).$$
Since $T_0$ has quantifier elimination, it follows that $\tp_0(L/K)=\tp_0(L'/K)$.

Now, by invariance, we have $L'\ind^0_F K$. Then, monotonicity and symmetry imply that $\phi(\sigma(a))\ind^0_F \phi(\sigma(a))$. By anti-reflexivity, $\phi(\sigma(a))\in \acl_0(F)\cap \UU_+=F$. Since $\phi\circ \sigma$ fixes $F$ pointwise, we get that $a\in F$, as desired.

\end{proof}

We now observe that in derivation-like theories with $T_0\subseteq T_+$ we have a natural description of $\dcl$.

\begin{lemma}\label{dcl}
Assume $\ind^0$ satisfies monotonicity, symmetry, full existence, and anti-reflexivity. In addition, assume that $T_0\subseteq T_+$. Then, for any $A\subset  \UU_+$, we have
$$\dcl_+(A)=\dcl_0(\langle A\rangle_\L).$$
\end{lemma}
\begin{proof}
As $T_0 \subseteq T_+$, we have
$$\dcl_0(\langle A\rangle_\L)\subseteq \dcl_+(A).$$
For the other containment, let $a\in \dcl_+(A)$ and let $\sigma$ be an $\L_0$-automorphism of $\UU_0$ fixing $\langle A\rangle_\L$ pointwise. We aim to show that $\sigma(a)=a$. Note that the assumption $T_0\subseteq T_+$ implies $\UU_+\preceq_{\L_0}\UU_0$; this together with Lemma~\ref{acl} yields
$$\sigma(\dcl_+(A))\leq_{\L_0} \sigma(\acl_0(\langle A\rangle_{\L})=\acl_0(\langle A\rangle_{\L})\leq_{\L_0} \UU_+.$$
Because $T$ is derivation-like, by Remark~\ref{aclA-and-useforCCM}(1), we have that $\sigma(\dcl_+(A))\leq_{\L}\UU_+$ and this is the unique $\L$-structure expanding its $\L_0$-structure, making it a model of $T^\forall$, and extending $\langle A\rangle_{\L}$. It follows from this and the fact that $T_+$ has quantifier elimination (by Corollary~\ref{corollary-qe-transfers}(3)), that $\sigma$ restricted to $\dcl_+(A)$ is a partial $\L$-elementary map of $\UU_+$. Thus, we may extend this restriction to an automorphism $\rho$ of $\UU_+$ (which fixes $A$). But then, as $a\in \dcl_+(A)$, we have that $a=\rho(a)=\sigma(a)$.
\end{proof}

\begin{remark}\label{almostder}
    We note that in the proofs of Lemmas~\ref{diagram},~\ref{acl}, and~\ref{dcl}, condition (ii) of derivation-like was only used when $A=B=C$. Namely, we applied Remark~\ref{aclA-and-useforCCM}(1), and so all results stated so far hold when $T$ is almost derivation-like with respect to $(T_0,\ind^0)$ in the sense of Remark~\ref{aclA-and-useforCCM}(2).
\end{remark}

\begin{definition}
Define the following relation on triples of small subsets of $\UU_+$:
\[
A \ind^+_C B \iff \acl_+(AC) \ind^0_{\acl_+(C)} \acl_+(BC).
\]
\end{definition}

\medskip

The following provides a detailed description of how independence properties of $\ind^0$ transfer to $\ind^+$.

\begin{theorem}\label{thm-ind-properties-transfer}
\begin{enumerate}
    \item $\ind^+$ is invariant and normal; 
    \item if $\ind^0$ satisfies any of monotonicity, symmetry, finite character, then so does $\ind^+$;
    \item if $\ind^0$ is transitive and monotone, then $\ind^+$ is transitive;
    \item if $\ind^0$ satisfies base monotonicity, finite character, and local character, then $\ind^+$ has local character;
    \item if $\ind^0$ satisfies normality, monotonicity, base monotonicity, transitivity, symmetry, and full existence, then $\ind^+$ satisfies base monotonicity;
    \item if $\ind^0$ satisfies monotonicity, symmetry, and full existence, then $\ind^+$ satisfies full existence;
    \item if $\ind^0$ satisfies monotonicity and extension and $T_+$ has quantifier elimination, then $\ind^+$ satisfies extension.
\end{enumerate}
\smallskip
In addition, if $T_0\subseteq T_+$, then {\normalfont (2), (3), and (7)} also hold for the corresponding properties stated over models.
\end{theorem}
\begin{proof}
\emph{Invariance}. Suppose $A \ind^+_C B$ and $\tp_+(ABC) = \tp_+(A'B'C')$. Then 
$$\tp_+(\acl_+(ABC)) = \tp_+(\acl_+(A'B'C')),$$
and similar arguments to the proofs above (using quantifier elimination for $T_0$) show that 
$$\tp_0(\acl_+(ABC)) = \tp_0(\acl_+(A'B'C')).$$ 
Invariance for $\ind^0$ then means that $A' \ind^+_{C'} B'$.

\emph{Normality} is by definition -- it does not require normality of $\ind^0$.

\emph{Monotonicity}. Suppose $A \ind^+_C B$ and $D \subseteq B$. Then $\acl_+(AC) \ind^0_{\acl_+(C)} \acl_+(BC)$. Also $\acl_+(DC) \subseteq \acl_+(BC)$, so by monotonicity for $\ind^0$, we have that 
$$\acl_+(AC) \ind^0_{\acl_+(C)} \acl_+(DC);$$
that is, $A \ind^+_C D$.

\emph{Transitivity} follows from transitivity and monotonicity of $\ind^0$.

\emph{Symmetry} follows from symmetry of $\ind^0$.

\emph{Finite character} follows from finite character of $\ind^0$ and the fact that $\acl_+$ is finitary.

\emph{Local character}. Precisely the same proof as in Theorem~2.1 of \cite{BloMar19} applies here.

\emph{Base monotonicity}. Suppose $A \ind^+_C B$ and $C \subseteq D \subseteq B$. We may also assume that $A \supseteq C$ by normality. Then $\acl_+(A) \ind^0_{\acl_+(C)} \acl_+(B)$. By monotonicity, we have $\acl_+(A) \ind^0_{\acl_+(C)} \acl_+(D)$. Since $T$ is derivation-like, $$\acl_0( \acl_+(A) \acl_+(D)) \cap \UU_+\leq_{\L}\UU_+.$$ 
So $\langle A D \rangle_\L \subseteq \acl_0( \acl_+(A) \acl_+(D)) \cap \UU_+$, and so by Lemma~\ref{acl} we have
$$\acl_+(AD) = \acl_0(\langle AD \rangle_\L) \cap \UU_+ \subseteq \acl_0(\acl_+(A) \acl_+(D)) \cap \UU_+.$$ 
By base monotonicity and normality for $\ind^0$, we get 
$$\acl_+(A) \acl_+(D) \ind^0_{\acl_+(D)} \acl_+(B).$$ 
By full existence, we get 
$$\acl_0(\acl_+(A) \acl_+(D)) \ind^0_{\acl_+(A) \acl_+(D)} \acl_+(B),$$ and by symmetry, transitivity, and monotonicity, $\acl_+(AD) \ind^0_{\acl_+(D)} \acl_+(B)$. That is, $A \ind^+_{D} B$.

For \emph{full existence}, suppose $a, A, B$ are given inside $\UU_+$. We need to find $a' \in \UU_+$ such that $\tp_+(a'/A) = \tp_+(a/A)$ and $a' \ind^+_A B$. Let $K$ be some small $\L$-elementary substructure of $\UU_+$ containing $a, A, B$. Write $C = \acl_+(A)$. Use full existence for $\ind^0$ to find $L \leq_{\L_0} \UU_0$ with $L \ind^0_{C} K$ and $\tp_0(L/C) = \tp_0(K/C)$. Let $\sigma \in \Aut(\UU_0/C)$ be the $\L_0$-automorphism taking $K$ to $L$. This automorphism then induces an $\L$-structure on $L$. Since $T$ is derivation-like, there is some $M \models T$ such that $K,L \leq_{\L} M \leq_{\L_0} \UU_0$. Since $T_+$ is the model companion of $T$, there is some $N \models T_+$ extending $M$. Let $\phi \colon N \to \UU_+$ be the $\L$-elementary embedding of $N$ inside $\UU_+$ that fixes $K$. Then
\[
\tp_+^{\UU_+}(a/C) = \tp_+^K(a/C) = \tp_+^L(\sigma(a)/C) = \tp_+^N(\sigma(a)/C) = \tp_+^{\UU_+}(\phi \sigma (a) /C).
\]

We also have the following chain of equalities of quantifier-free $\L_0$-types:
\[
\qftp_0^{\UU_0}(L/K) = \qftp_0^M(L/K) = \qftp_0^N(L/K) = \qftp_0^{\UU_+}(\phi(L)/K) = \qftp_0^{\UU_0}(\phi(L)/K).
\]



As $T_0$ has quantifier elimination, this yields $\tp_0(L/K)=\tp_0(\phi(L)/K)$. Invariance then gives $\phi(L) \ind^0_C K$, and monotonicity gives $\acl_+(C, \phi \sigma(a)) \ind^0_C \acl_+(AB)$; that is, $\phi \sigma(a) \ind^+_A B$.

\emph{Extension.} Suppose $A \ind^+_C B$ and $D \supseteq B$ is given. We need to find $A' \equiv_{BC} A$ with $A' \ind^+_C D$. As usual we may assume $C \subseteq A, B$ and that these parameter sets are $\acl_+$-closed. Let $K$ be a small $\L$-elementary substructure of $\UU_+$ containing all of these sets. Use extension for $\ind^0$ to find $A' \ind^0_{C} K$ with $\tp_0(A'/ B) = \tp_0(A/ B)$. This $\L_0$-isomorphism induces an $\L$-isomorphic structure on $A'$. By the derivation-like axiom, there is some $M \models T$ such that $M \leq_{\L_0} \UU_0$ with $A', K \leq_\L M$. Since $T_+$ is the model companion of $T$, extend $M$ to some $N \models T_+$, and let $\phi \colon N \to \UU_+$ be an $\L$-elementary embedding of $N$ in $\UU_+$ which fixes $K$. Then
\[
\qftp_+^{\UU_+}(A/B) = \qftp_+^K(A/B) = \qftp_+^M(A'/B) = \qftp_+^N(A'/B) = \qftp_+^{\UU_+}(\phi (A') /B).
\]
By quantifier elimination for $T_+$, $\tp_+(A/B) = \tp_+(\phi(A') / B)$.

As usual,
\[
\qftp_0^{\UU_0}(A'/K) = \qftp_0^{M}(A'/K) = \qftp_0^{N}(A'/K) = \qftp_0^{\UU_+}(\phi(A')/K) = \qftp_0^{\UU_0}(\phi(A')/K),
\]
and by quantifier elimination for $T_0$ and invariance and monotonicity for $\ind^0$, we get $\phi(A') \ind^0_C D$.

\medskip

For the final clause of the statement, note that the arguments provided in (2), (3), and (7) hold when working over models since $T_0 \subseteq T_+$. As an example, we will give the proof that if $\ind^0$ satisfies monotonicity over models, then $\ind^+$ satisfies monotonicity over models.

\emph{Monotonicity over models.} Suppose $A \ind^+_C B$ and $D \subseteq B$ where $C \models T_+$. Then $\acl_+(AC) \ind^0_{C} \acl_+(BC)$ -- since $C$ is $\acl_+$-closed -- and $\acl_+(DC) \subseteq \acl_+(BC)$. Now since $T_0 \subseteq T_+$, $C$ is also a model of $T_0$, and so by monotonicity over models for $\ind^0$, we have that 
$$\acl_+(AC) \ind^0_{C} \acl_+(DC);$$
that is, $A \ind^+_C D$.
\end{proof}

Using the above theorem, we observe that rosiness transfers.

\begin{corollary}\label{rosy-transfers}
Assume both $T_0$ and $T_+$ admit elimination of imaginaries. If $T_0$ is rosy, then so is $T_+$.
\end{corollary}
\begin{proof}
    By Theorem~\ref{neostability}(iv), it suffices to show that $T_+$ admits an strict independence relation. Taking $\ind^0$ to be any strict independence relation (which exists by rosiness of $T_0$), Theorem~\ref{thm-ind-properties-transfer} yields that $\ind^+$ is an independence relation; thus, it suffices to show that $\ind^+$ satisfies anti-reflexivity. Suppose $a\ind^+_C a$. Then, by symmetry and monotonicity of $\ind^0$, we get $a\ind^0_{\acl_+(C)} a$; and so, by anti-reflexity of $\ind^0$, we obtain 
    $$a\in \acl_0(\acl_+(C))\cap \UU_+=\acl_+(C),$$
    where the last equality uses Lemma~\ref{acl}.
\end{proof}

We now address the transfer of the independence theorem.

\begin{theorem}\label{thm-ind-thm-transfers}
Let $M \models T_+$ and suppose $T_0\subseteq T$. Assume the following: $T_0$ has, in addition to $\ind^0$, an independence relation $\ind^1$ such that $\ind^0_M \implies \ind^1_M$; $T$ is derivation-like with respect to $(T_0, \ind^1)$; and $\ind^0_M$ satisfies monotonicity, symmetry, and extension. 
If $\ind^0$ satisfies the independence theorem over $M$, then so does $\ind^+$.
\end{theorem}

\begin{proof}
Let $M \models T_+$, $A_1 \ind^+_M A_2$, $a_1 \ind^+_{M} A_1$, $a_2 \ind^+_{M} A_2$, and $\tp_+(a_1/M) = \tp_+(a_2/M)$. We will show that there is $a \ind^+_{M} A_1 A_2$ realising $\tp_+(a_1/MA_1) \cup \tp_+(a_2/MA_2)$. Let $N \models T_+$ be some $\L$-elementary substructure of $\UU_+$ containing all of the above subsets. 

Note that, by Theorem~\ref{thm-ind-properties-transfer}, $\ind^+_M$ satisfies symmetry and extension (the fact that $T$ is derivation-like with respect to the independence relation $\ind^1$ yields that $T_+$ has quantifier elimination by Corollary~\ref{corollary-qe-transfers}). Thus, we may assume that $A_1$, $A_2$, $a_1$, and $a_2$ are all $\acl_+$-closed and contain $M$ (note that $T_0\subseteq T$ implies that they are also $\acl_0$-closed).

\medskip

\textbf{Claim 1.} There is some $a \in \UU_0$ with $a \ind^0_{M} N$ and $a \models \tp_0(a_1/A_1) \cup \tp_0(a_2/A_2)$.\\
\noindent \textit{Proof of claim.} Note first that by definition of $\ind^+$, we have the following facts: $A_1 \ind^0_M A_2$, $a_1 \ind^0_{M} A_1$, $a_2 \ind^0_{M} A_2$, and $\tp_0(a_1/M) = \tp_0(a_2/M)$. By the independence theorem for $\ind^0$, there is $a \in \UU_0$ with $a \ind^0_{M} A_1 A_2$ and $a \models \tp_0(a_1/A_1) \cup \tp_0(a_2/A_2)$. Now by extension for $\ind^0_M$, we may assume that $a \ind^0_{M} N$.

\medskip

\textbf{Claim 2.} Inside $\UU_0$, there are $\L_0$-isomorphic copies of $N$, $N_1$ and $N_2$, both containing $a$, with $N_1 \ind^1_{a} N_2$ and $N \ind^1_{A_1 A_2} N_1 N_2$.\\
\noindent \textit{Proof of claim.} Note first that, by assumption and the fact that $a \ind^0_M N$, we have that $a \ind^1_M N$. Now for $i=1,2$, let $N_i'$ be the copy of $N$ coming from the $\L_0$-automorphism $A_i a_i \mapsto A_i a$. By full existence for $\ind^1$, let $N_i \equiv_{A_i a}^0 N_i'$ with $N_1 \ind^1_{A_1 a} N$ and $N_2 \ind^1_{A_2 a} NN_1$. Then $N \ind^1_{A_1} N_1$ and $N \ind^1_{A_2} N_2$ by transitivity. From $a \ind^1_{M} N$ we get $a \ind^1_{A_1} A_2$, and so $A_1 a \ind^1_{A_1} A_2$. Along with $A_1 \ind^1_{M} A_2$, transitivity gives $A_1 a \ind^1_{M} A_2$, so that $A_1 a \ind^1_{a} A_2$ by base monotonicity. This implies $A_1 \ind^1_{a} A_2$ and $N_1 \ind^1_{a} A_2$. This last part implies $N_1 \ind^1_{a} A_2 a$ and along with $N_2 \ind^1_{A_2 a} N N_1$ implies $N_1 \ind^1_{a} N_2$. Also, $N \ind^1_{A_1 A_2} N_1$ by base monotonicity since $A_1 A_2 \subseteq N$. From $N N_1 \ind^1_{A_2 a} N_2$, we get $N \ind^1_{A_2 a N_1} N_2$, and hence $N \ind^1_{A_2 N_1} N_2$ since $a \in N_1$. Combining this with $N \ind^1_{A_1 A_2} N_1$ gives $N \ind^1_{A_1 A_2} N_1 N_2$.

\medskip

\textbf{Claim 3.} There is some model of $T$ which is an $\L$-extension of $N$, $N_1$, and $N_2$.\\
\noindent \textit{Proof of claim.} Define $\L$-structures on $N_1$ and $N_2$ such that $(N_i, A_i, a)$ is $\L$-isomorphic to $(N, A_i, a_i)$. 
So $N_i \models T_+$ for $i=1,2$. Note that since $a_i$ is an $\L$-substructure of $N$, $a$ is also an $\L$-substructure of $N_i$. Now $N_1 \ind^1_{a} N_2$, and $a$ is relatively $\acl_0$-closed in $N_1$ and $N_2$; since $T$ is derivation-like with respect to $(T_0,\ind^1)$, there is some $P \models T$ such that $P \leq_{\L_0} \UU_0$ and $N_1, N_2 \leq_\L P$. 
Since $T_0\subseteq T$ and using part (ii) of the definition of derivation-like, we have $\acl_0(N_1N_2)$ is an $\L$-substructure of $P$. By the uniqueness clause of part (ii) of derivation-like and the fact that $A_1 \ind^1_{M} A_2$, we have that $\acl_0(A_1 A_2)$ is equipped with an $\L$-structure that makes it simultaneously an $\L$-substructure of $N$ and an $\L$-substructure of $\acl_0(N_1N_2)$. Now $N \ind^1_{A_1 A_2} N_1 N_2$, and so $N \ind^1_{\acl_0(A_1A_2)} \acl_0(N_1N_2)$ by invariance, base monotonicity, monotonicity, transitivity, and full existence. Now by part (i) of the derivation-like axiom we may find some $S \models T$ with $S \leq_{\L_0} \UU_0$ and $N, \acl_0(N_1N_2) \leq_\L S$. So $S$ is an $\L$-extension of $N$, $N_1$, and $N_2$, as desired.

\medskip

Now $S$ extends to some $S' \models T_+$. Let $j \colon S' \to \UU_+$ be an $\L$-elementary embedding of $S'$ in $\UU_+$ that fixes $N$. Then
\[
\tp_+^{\UU_+}(a_1/A_1) = \tp_+^N(a_1/A_1) = \tp_+^{N_1}(a/A_1) = \tp_+^{S'}(a/A_1) = \tp_+^{\UU_+}(j(a)/A_1)
\]
and similarly we have $j(a) \equiv_{A_2}^+ a_2$. 

As $T_0$ has quantifier elimination, we have $\tp_0(a/N)=\tp_0(j(a)/N)$. Now, by construction of $a$, we had $a \ind^0_{M} N$, and by monotonicity and invariance, we get $j(a) \ind^0_{M} \acl(A_1 A_2)$, and so $j(a) \ind^+_{M} A_1 A_2$.
\end{proof}

The following addresses transfer of stationarity.

\begin{theorem}\label{thm-stationarity-transfers}
Let $M \models T_+$ with $\dcl_0(M)\models T_0$. Suppose $\ind^0$ satisfies base monotonicity, extension, and full existence. If $\ind^0$ satisfies stationarity over $\dcl_0(M)$, then $\ind^+$ satisfies stationarity over $M$.
\end{theorem}

\begin{proof}
Note that, by full existence and Corollary~\ref{corollary-qe-transfers}, $T_+$ has quantifier elimination.
Now suppose $M \subset N\subset \UU_+$, $a, b \in \UU_+$ with $\tp_+(a/M) = \tp_+(b/M)$, $a \ind^+_M N$, and $b \ind^+_M N$. Since $\ind^+$ satisfies extension (by Theorem~\ref{thm-ind-properties-transfer}(7)), we may assume that $N$ is a model of $T_+$. Let $K_a = \acl_+(Ma)$ and $K_b = \acl_+(Mb)$. By definition of $\ind^+$, we have that $K_a \ind^0_M N$ and $K_b \ind^0_M N$. By extension for $\ind^0$, the same independence holds after replacing $N$ for some $N_0$ containing $N\cup \dcl_0(M)$. Hence, by base monotonicity, $K_a \ind^0_{\dcl_0(M)} N_0$ and $K_b \ind^0_{\dcl_0(M)} N_0$. Note that $\tp_0(K_a / \dcl_0(M)) = \tp_0(K_b / \dcl_0(M))$. Then by stationarity for $\ind^0$, 
\[
\tp_0(K_a / N) = \tp_0(K_b / N).
\]
This implies that there is an $\L_0$-isomorphism $\langle K_a N \rangle_{\L_0} \to \langle K_b N \rangle_{\L_0}$ taking $a \mapsto b$ and fixing $N$.

Note that $M$ is an $\L$-substructure of $K_a$, $K_b$, and $N$. Furthermore, by Remark~\ref{aclA-and-useforCCM}(1) and full existence, $\acl_0(M)\cap \UU_+=\acl_+(M)=M$ (as $M\models T_+$).


By the derivation-like axiom, $\langle K_a N \rangle_{\L_0}$ and $\langle K_b N \rangle_{\L_0}$ are $\L$-substructures of $\UU_+$. By its uniqueness clause, this $\L_0$-isomorphism must be an $\L$-isomorphism. So $\qftp_+(a/N) = \qftp_+(b/N)$. By quantifier elimination for $T_+$, we have $\tp_+(a/N) = \tp_+(b/N)$.
\end{proof}

\begin{corollary}\label{neostability-transfers}
Suppose $\ind^0$ is nonforking independence.
\begin{enumerate}
    \item Assume that $\dcl_0(M)\models T_0$ whenever $M\models T_+$. If $T_0$ is stable, then $T_+$ is stable and $\ind^+$ is nonforking independence.
    \item Assume $T_0\subseteq T$. If $T_0$ is simple, then $T_+$ is simple and $\ind^+$ is nonforking independence.
\end{enumerate}
\end{corollary}
\begin{proof}
    (1) follows from Theorems~\ref{thm-ind-properties-transfer} and \ref{thm-stationarity-transfers}; while (2) follows from Theorems~\ref{thm-ind-properties-transfer} and \ref{thm-ind-thm-transfers} (note that in the latter we take $\ind^1=\ind^0$).
\end{proof}

\begin{example}
In Corollary~\ref{neostability-transfers} we required that $\dcl_0(M) \models T_0$ for every $M \models T_+$. If $T_0 \subseteq T_+$, then this holds automatically since $T_0$ is model complete. Otherwise, Example~\ref{example-T0-in-T+} provides a case where this requirement holds. Suppose $T_0 = \ACF_p$ and $T = \SCF_{p,e}$. Then for any set $A$, $\dcl_0(A)$ coincides with the perfect closure of the field generated by $A$. So, if $M \models \SCF_{p,e}$, then $\dcl_0(M) \models \ACF_p$.
\end{example}

We now aim to prove a similar result on the transfer of NSOP$_1$. We will need to restrict to the case of fields to apply Theorem~\ref{thm-ind-thm-transfers} with a particular choice of $\ind^1$.

Assume $T_0$ is an $\L_0$-theory of fields, we say that a relation $\ind^1$ on $T_0$ \emph{implies $\L_0$-compositums} if for all $K, L\leq_{\L_0}\UU_0$ satisfying $K\ind_E^1 L$, for some $E=\acl_0(E)\cap K=\acl_0(E)\cap L$, the compositum $K\cdot L$ is an $\L_0$-substructure of $\UU_0$. Following \cite{JunKo2010}, we say that $T_0$ is \emph{very $\L_0$-slim} if for every $F\leq_{\L_0}\UU_0$ we have that 
$$\acl_0(F)=F^{\alg}\cap \UU_0.$$

Define the relation $\ind^1$ on small subsets of $\UU_0$ by
\begin{equation}\label{defalgdis1}
A\ind^1_C B \quad \iff \quad \langle A\, C\rangle_{\L_0}\ind^{\alg}_{\langle C\rangle_{\L_0}} \langle B\, C\rangle_{\L_0}
\end{equation}
where $\ind^{\alg}$ denotes algebraic independence in fields.


\begin{fact}\label{fact-l0-compositumns}
Assume $\ind^1$ implies $\L_0$-compositums. The relation $\ind^1$ as defined above is an independence relation if and only if $T_0$ is very $\L_0$-slim.
\end{fact}

The proof is an adaptation of Theorem 2.1 of \cite{JunKo2010}. Some details are provided in Lemma 4.4.7 of the second author's thesis \cite{Moha24}.

\begin{corollary}\label{corollay-nsop-transfers}
Assume that $T_0$ is very $\L_0$-slim, that $\ind^1$ implies $\L_0$-compositums, that $T$ is derivation-like with respect to $(T_0, \ind^1)$, and that $T_0\subseteq T$. If $T_0$ is NSOP$_1$ and $\ind^0$ is Kim-independence, then $T_+$ is NSOP$_1$ and $\ind^+$ is Kim-independence.
\end{corollary}
\begin{proof}
By Proposition~3.9.26 of \cite{Ram18}, if two subfields are Kim-independent over a submodel, then they are algebraically independent. So the condition $\ind^0_M \implies \ind^1_M$ holds, and $\ind^+$ satisfies the independence theorem over models by Theorem~\ref{thm-ind-thm-transfers} (noting that $\ind^1$ is an independence relation by Fact~\ref{fact-l0-compositumns}). Existence over models and chain local character each transfer from $\ind^0$ to $\ind^+$ since every model of $T_+$ is also a model of $T_0$. The remaining conditions of Theorem~\ref{neostability}(iii) hold by Theorem~\ref{thm-ind-properties-transfer} (note that (7) of that theorem, the transfer of existence, does apply as $T_+$ has quantifier elimination; this is by Corollary~\ref{corollary-qe-transfers} and the fact that $\ind^1$ satisfies full existence).
\end{proof}

\bigskip

\section{Examples}\label{examples}

In this section we observe that there are plenty of examples of theories that are derivation-like, and hence to which the results of the previous section apply (when the model companion exists).

\medskip

\subsection{Separably closed fields and Hasse-Schmidt fields}

Fix a prime $p>0$ and $e$ a finite non-negative integer. Let $\L_0$ be the language of fields, $T_0=\ACF_p$ and $\ind^0$ forking independence (which coincides with algebraic disjointness $\ind^{\alg}$). Let $\L_{b,\lambda}$ be the language of fields expanded by constants $b=(b_1,\dots,b_e)$ and unary function symbols $(\lambda_i)_{i\in p^e}$. Let $T_{b,\lambda}$ be the theory of fields of characteristic $p$ together with sentences specifying that $b$ is a $p$-basis and that the $\lambda_i$'s are interpreted as the $\lambda$-functions with respect to $b$ (in some fixed order of the $p$-monomials).

\begin{lemma}\label{oplikeSCF}
The theory $T_{b,\lambda}$ is derivation-like with respect to $(\ACF_p,\ind^{\alg})$.
\end{lemma}
\begin{proof}
With $\UU_0$ a monster model of $\ACF_p$, let $A,B,C\models T_{b,\lambda}^{\forall}$ be as in the definition of derivation-like. Since $C\leq_{\L_{b,\lambda}}A$, we have that $A/C$ is a separable field extension. This, together with the fact that $C=C^{\alg}\cap A$, implies that the field extension $A/C$ is regular (i.e., separable and relatively algebraically closed). This, together with $A\ind_C^{\alg} B$, implies that $A$ and $B$ are linearly disjoint over $C$. Linear disjointness implies that $b$ is $p$-independent in the compositum $A\cdot B$ (and hence a $p$-basis). If follows that $A\cdot B\models T_{b,\lambda}$ and $A,B\leq_{\L_{b,\lambda}}A\cdot B\leq_{\L_0} \UU_0$. This shows condition (i) of derivation-like. Condition (ii) follows from the fact that $p$-bases are preserved when passing to separably algebraic extensions.
\end{proof}

The model companion of $T_{b,\lambda}$ is $\SCF_{p,e}$. Note that in this case for any $M\models \SCF_{p,e}$ we have that $\dcl^{\ACF_p}(M)\models \ACF_p$ (since perfect closures of separably closed fields remain separably closed). Thus, our Corollary~\ref{neostability-transfers}(1) applies and recovers the well known fact that $\SCF_{p,e}$ is stable and (in the language $\L_{b,\lambda}$) forking independence coincides with algebraic disjointness.

\medskip

In a similar fashion we can also recover the context of iterative Hasse-Schmidt derivations from \cite{Zieg03}. Let $\L_{\partial}$ be the language of fields expanded by unary function symbols
$$((\partial_{1,j})_{j=1}^{\infty},\dots,(\partial_{e,j})_{j=1}^\infty).$$
Let $T_{\partial}$ be the theory of fields of characteristic $p$ expanded by sentences specifying that $(\partial_{i,j})_{j=1}^\infty$ is an iterative Hasse-Schmidt derivation and that, for different $i$, they pairwise commute.

\begin{lemma}
The theory $T_{\partial}$ is derivation-like with respect to $(\ACF_p,\ind^{\alg})$.
\end{lemma}
\begin{proof}
Let $A,B,C\models T_{\partial}^{\forall}$ be as the definition of derivation-like. By Lemma 2.3 and 2.4 from \cite{Zieg03}, after possibly passing to a purely inseparable extension of the separable closure of $C$, we may assume that $C$ is strict and separably closed. Strictness implies that $A/C$ is a separable extension. Thus, since $C$ is separably closed, $A/C$ is a regular field extension; this implies that $A$ and $B$ are linearly disjoint over $C$. It follows that $A\cdot B$ is isomorphic to the quotient field of $A\otimes_C B$. The Hasse-Schmidt derivations extend uniquely to $A\otimes_C B$ and this yields an $\L_\partial$-structure on $A\cdot B$ making it a model of $T_\partial$ (see for instance Lemma 2.5 of \cite{Zieg03}). This yields condition (i) of derivation-like. Since Hasse-Schmmidt fields have a smallest strict extension (see \cite[Lemma 2.4]{Zieg03}) and separably algebraic extensions are \'etale, condition (ii) of derivation-like follows.
\end{proof}

The model companion of $T_{\partial}$ is $\SCH_{p,e}$ (using the notation from \cite{Zieg03}). Recall that the latter is the theory of fields equipped with strict and iterative Hasse-Schmidt derivations that pairwise commute and whose underlying field is a model of $\SCF_{p,e}$. As in the $\SCF$ case above, for any $M\models \SCH_{p,e}$ we have that $\dcl^{\ACF_p}(M)\models \ACF_p$. Thus, Corollary~\ref{neostability-transfers}(1) applies and recovers the fact that $\SCH_{p,e}$ is stable and in the language $\L_{\partial}$ forking independence coincides with algebraic disjointness.

\

\subsection{$\D$-fields in characteristic zero}

Let $\L_0$ be an expansion of the field language and $T_0$ a complete and model complete $\L_0$-theory of fields of characteristic zero. As before, we denote by $\ind^0$ an invariant ternary relation on a monster model $\UU_0\models T_0$. Recall that $\ind^{\alg}$ denotes the algebraic disjointness relation. 

We say that $\ind^0$ \emph{implies algebraic disjointness} if $$K\ind^0_E L\implies K\ind^{\alg}_E L$$
for $K,L$ $\L_0$-substructures of $\UU_0$ and $E$ a common $\L_0$-substructure of $K$ and $L$. Recall from the previous section that $\ind^0$ \emph{implies $\L_0$-compositums} if for all $K, L\leq_{\L_0}\UU_0$ satisfying $K\ind_E^0 L$, for some $E=\acl_0(E)\cap K=\acl_0(E)\cap L$, the compositum $K\cdot L$ is an $\L_0$-substructure of $\UU_0$. Recall also that $T_0$ is \emph{very $\L_0$-slim} if for every $F\leq_{\L_0}\UU_0$ we have that 
$$\acl_0(F)=F^{\alg}\cap \UU_0.$$

Following the general framework of Moosa-Scanlon from \cite{MooScan14}, let $\D$ be a finite-dimensional algebra over a field $k$ of characteristic zero equipped with a $k$-basis $\epsilon_0=1,\epsilon_1,\dots,\epsilon_d$ such that $\D$ is a local ring with residue field $k$. A $\D$-field $K$ is a field which is also a $k$-algebra equipped with a sequence of operators $(\partial_i:K\to K)_{i=1}^d$ such that the map $K\to K\otimes_k\D$ defined by
$$a\mapsto a\otimes \epsilon_0 + \partial_1(a)\otimes\epsilon_1+\cdots \partial_d(a)\otimes\epsilon_d$$
is a $k$-algebra homomorphism. Let $\L_\D$ be the language $\L_0$ expanded by the language of $k$-algebras and the unary function symbols $\{\partial_1,\dots,\partial_d\}$. Let $T_\D$ be $\L_\D$-theory consisting of $T_0$ together with the theory of $\D$-fields. In addition, let $T_{\D^*}$ be $T_\D$ expanded by sentences specifying that the $\partial_i$'s pairwise commute.

\begin{remark}\label{differential}
Let $\D=\mathbb Q[x_1,\dots,x_d]/(x_1,\dots,x_d)^2$. In this case, the theory $T_\D$ is the theory of differential fields of characteristic zero with $d$ many (not necessarily commuting) derivations whose underlying field is a model of $T_0$. The theory $T_{\D^*}$ is similar but requires the derivations to pairwise commute.
\end{remark}

\begin{lemma}
Suppose $\ind^0$ implies algebraic disjointness and $\L_0$-compositums. Also assume $T_0$ is very $\L_0$-slim. Then, the theories $T_\D$ and $T_{\D^*}$ are derivation-like with respect to $(T_0,\ind^0)$.
\end{lemma}
\begin{proof}
    First we prove $T_\D$ is derivation-like. Let $K,L,E\models T_\D^\forall$ be as in the definition of derivation-like. Since $E=\acl_0(E)\cap K$ and $T_0$ is very $\L_0$-slim, $K/E$ is a regular field extension. Since $\ind^0$ implies algebraic disjointness, it follows that $K$ and $L$ are linearly disjoint over $E$. Then $K\cdot L$ is isomorphic to the quotient field of $K\otimes_E L$. Since $\ind^0$ implies $\L_0$-compositums and $\D$-structures extend uniquely to $K\otimes_E L$ (see \cite[Proposition 2.20]{BHKK19}), this yields an $\L_\D$-structure on $K\cdot L$. As we are in characteristic zero, this $\D$-structure extends to all of $\UU_0$ (recall that $\D$-structures always extend to smooth extensions, see \cite[Lemma 2.7(3)]{BHKK19}) which yields condition (i) of derivation-like. Since algebraic extensions are \'etale (in characteristic zero), condition (ii) follows (recall that $\D$-structures extend uniquely to \'etale extensions, see \cite[Lemma 2.7(2)]{BHKK19}).

For $T_{\D^*}$, the same argument works by simply noting that uniqueness of the $\D$-structure on $K\otimes_E L$ forces the $\partial_i$'s to commute. And similarly when passing to algebraic extensions (as they are \'etale in characteristic zero). To extend the $\D$-structure from $K\cdot L$ to $\UU_0$, first extend to a transcendence basis in a trivial way to force commutativity of the $\partial_i$'s and after this the unique extension to $\UU_0$ will necessarily commute. This sort of argument is spelled out in Example 4.4.11 of the second author's thesis \cite{Moha24}. 
\end{proof}

For the remainder of this section we set $\ind^0$ to be the relation we defined in (1) at the end of Section~\ref{mainresults}.
\begin{equation}\label{defalgdis}
A\ind^0_C B \quad \iff \quad \langle A\, C\rangle_{\L_0}\ind^{\alg}_{\langle C\rangle_{\L_0}} \langle B\, C\rangle_{\L_0}
\end{equation}

Note that this particular relation implies algebraic disjointness. We recall Fact~\ref{fact-l0-compositumns} along with an additional fact.

\begin{fact}
    \begin{enumerate}
    \item Assume $\ind^0$ implies $\L_0$-compositums. The relation $\ind^0$ as defined in \eqref{defalgdis} is an independence relation if and only if $T_0$ is very $\L_0$-slim.
    \item Suppose $\L_0=\L_{fields}(C)$ where $C$ is a set of constant symbols. If models of $T_0$ are large fields, then $T_0$ is very $\L_0$-slim.
    \end{enumerate}
\end{fact}

We note that (2) is an adaptation of Theorem 5.4 of \cite{JunKo2010} and appears in Lemma 4.4.10 of the second author's thesis \cite{Moha24}.

\begin{corollary}\label{specialcase}
    Suppose models of $T_0$ are large fields, $\L_0=\L_\text{fields}(C)$ for $C$ a set of constant symbols, and $\ind^0$ is given as in \eqref{defalgdis}. Assume $T_\D$ and $T_{\D^*}$ have model companions $T_\D^+$ and $T_{\D^*}^+$, respectively.
    \begin{itemize}
        \item [(i)] If $T_0$ is simple, then so are $T_\D^+$ and $T_{\D^*}^+$.
        \item [(ii)] If $T_0$ is NSOP$_1$, then so are $T_\D^+$ and $T_{\D^*}^+$.
        \item [(iii)] Assume $T_0$, $T_\D^+$ and $T_{\D^*}^+$ all eliminate imaginaries. If $T_0$ is rosy, then so are $T_\D^+$ and $T_{\D^*}^+$.
    \end{itemize}
\end{corollary}
\begin{proof}
(i) follows immediately by Corollary~\ref{neostability-transfers}(2), (ii) by Corollary~\ref{corollay-nsop-transfers}, (iii) by Corollary~\ref{rosy-transfers}.
\end{proof}


\begin{remark}\label{applications}
    Under the hypothesis of Corollary~\ref{specialcase} (and recall that we are in characteristic zero), an immediate application is that
    \begin{itemize}
        \item [(i)] if $T_0$ is the theory of a bounded PAC field , then $T_\D^+$ and $T_{\D^*}^+$ are simple,
        
        \item [(ii)] the theory CODF (closed ordered differential field in one derivation) is rosy.
    \end{itemize}
    Indeed, for (i) recall that bounded PAC fields are simple; while for (ii), recall that the theory RCF is rosy and eliminates imaginaries; also, CODF eliminates imaginaries by \cite{CKP23}.
\end{remark}

We also note that under the hypothesis of Corollary~\ref{specialcase} the model companions $T_\D^+$ and $T_{\D^*}^+$ in fact exist. Existence of $T_\D^+$ is one of the main results of the second author's paper \cite{Moha23} and, moreover, the simplicity claimed in Remark~\ref{applications}(i) already appears there. Existence of $T_{\D^*}^+$ will appear in a forthcoming paper.\footnote{ As part of joint work of the first author with Jan Dobrowolski.} In the case when $\D=\mathbb Q[x_1,\dots, x_d]/(x_1,\dots x_d)^2$ (i.e., in the context of differential fields of characteristic zero, see Remark~\ref{differential}) and $T_0$ is the theory of a bounded PAC field, the existence of $T_{\D^*}^+$ is an instance of the main result of \cite{Tre05} and its simplicity already appears in \cite{HoffLS22}. 

A natural question is whether Corollary~\ref{corollay-nsop-transfers} applies when $T_0$ is the theory of an $\omega$-free PAC field of characteristic $0$. The authors do not know if such a theory can be made model complete in a language $\L_\text{fields}(C)$ for some set of constant symbols $C$, which is required for Corollary~\ref{specialcase} and to argue as above that
$T_\D^+$ and $T_{\D^*}^+$ exist. However, in the case $\D=\mathbb Q[x_1,\dots, x_d]/(x_1,\dots x_d)^2$ -- the context of differential fields of characteristic zero, see Remark~\ref{differential} -- we may use results of Fornasiero and Terzo \cite{ForTer24} on 
algebraically bounded structures with generic derivations. Any PAC field of characteristic zero is very $\L_\text{fields}$-slim by Corollary~4.5 
of \cite{CP98}. In Section~30.2 of \cite{FriJar86}, the authors expand by definitions an $\omega$-free PAC field to a language $\L_\text{fields}(R_n \colon n > 0)$ and show that, in this language, $T_0$ is model complete. Here the symbols $R_n$ are predicates that ensure that extensions are regular field extensions. This theory is then very $\L_\text{fields}(R_n: n>0)$-slim, and hence algebraically bounded. By Theorems~4.5 and~5.12 of \cite{ForTer24}, $T_\D^+$ and $T_{\D^*}^+$ exist. By Corollary~\ref{corollay-nsop-transfers} of this paper, both are NSOP$_1$.

\

\subsection{Differential fields in positive characteristic}

In this section we apply our results in the context of separably differentially closed fields \cite{InoLS23}. Fix a prime $p>0$. Let $\L_\lambda$ be the language of fields expanded by the (countably many) $\lambda$-functions. Namely, by functions $(\lambda_{n,i}:n\in \mathbb \omega, i\in p^n)$ where $\lambda_{n,i}$ is $(n+1)$-ary. We denote by $\SCF_{p,\infty}^\lambda$ the theory of separably closed fields of infinite (algebraic) degree of imperfection expanded by sentences specifying that the $\lambda_{n,i}$ are to be interpreted as the $\lambda$-functions; that is, for $(\bar a,b)$, if $\bar a$ is $p$-dependent or $(\bar a,b)$ is $p$-independent then $\lambda_{n,i}(\bar a,b)=0$; otherwise, 
$$b=\sum_{i\in p^n}(\lambda_{n,i}(\bar a,b))^p\, m_i(\bar a)$$
where the $m_i(\bar a)$'s denote the $p$-monomials (with some fixed order).
We denote forking independence in $\SCF_{p,\infty}^\lambda$ by $\ind^0$, and $\UU_0$ is a monster model. Recall from \cite{Srour86} that for $\L_\lambda$-substructures $K,L,E$ of $\UU_0$ we have
$$K\ind^0_E L \quad \iff \quad \text{$K$ and $L$ are algebraically disjoint and $p$-disjoint over $E$.}$$

Let $\L_{\lambda,\delta}$ be the expansion of $\L_\lambda$ by a (single) unary function symbol $\delta$ and let $\DF_p^\lambda$ be the theory of differential fields of characteristic $p$ with sentences specifying that the $\lambda_{n,i}$ are the $\lambda$-functions.

\begin{lemma}
The theory $\DF_p^\lambda$ is derivation-like with respect to $(\SCF_{p,\infty}^\lambda, \ind^0)$.
\end{lemma}
\begin{proof}
Let $K,L,,E\models (\DF_{p}^\lambda)^\forall$ be as in the definition of derivation-like. Since $K\ind_E^{0} L$, $K$ and $L$ are $p$-disjoint over $E$, and so $K\cdot L$ is an $\L_\lambda$-substructure of $\UU_0$.

On the other hand, since $E\leq_{\L_{\lambda,\delta}}K$, we have that $K/E$ is a separable field extension. This, together with the fact that $E=E^{\alg}\cap K$, implies that $K/E$ is a regular field extension. This, together with $K\ind_E^{0} L$, implies that $K$ and $L$ are linearly disjoint over $E$. Linear disjointness implies $K\cdot L$ is isomorphic to the quotient field of $K\otimes_E L$. This yields a derivation on $K\cdot L$ making it a model of $\DF_{p}^\lambda$. This yields condition (i) of derivation-like. Since separably algebraic extensions are \'etale, condition (ii) follows.
\end{proof}

For $\epsilon\in \mathbb N\cup \{\infty\}$, recall that a differential field $(K,\delta)$ of characteristic $p$ is said to have differential degree of imperfection $\epsilon$ if 
$$[C_K:K^p]=p^\epsilon.$$
Here $C_K$ denotes the field of $\delta$-constants of $(K,\delta)$. When $\epsilon=\infty$ the above equality should be understood as the degree $[C_K:K^p]$ being infinite. See \cite{InoLS23} for further details.

In \cite{InoLS23} it was shown that $\DF_p^{\lambda}$ has a model companion; namely, the theory $\SDCF_{p,\infty}^\lambda$ of separably differentially closed fields of characteristic $p$ of infinite differential degree of imperfection expanded by the $\lambda$-functions. We note that in \cite{InoLS23} the authors work in the language of the so-called differential $\lambda$-functions, denoted $\ell_{n,i}$, but the result on the existence of the model companion also holds working with the algebraic $\lambda$-functions (the argument is spelled out in \cite[Fact 4.4.16]{Moha24}).  We note that $\SCF_{p,\infty}^\lambda\subseteq \SDCF_{p,\infty}^\lambda$. Thus, Corollary~\ref{neostability-transfers}(1) applies and recovers the fact that $\SDCF_{p,\infty}^\lambda$ is stable; furthermore, it shows that, in the language $\L_{\lambda,\delta}$, forking independence coincides with algebraic disjointness and $p$-disjointness (which is not explicitly stated in \cite{InoLS23}).

\bigskip

We conclude this section by noting that, unfortunately, our results do not seem to apply to the theory $\DCF_p$, differentially closed fields of characteristic $p>0$, studied by Wood in \cite{Wood73}. Recall that $\DCF_p$ is the model companion of $\DF_p$, the theory of differential fields (of characteristic $p$) in the language of differential fields $\L_{\delta}$. In this language the theory $\DCF_p$ does not eliminate quantifiers, but in \cite{Wood73} Wood showed that it suffices to add the \emph{$p$-th root on constants} function $\ell_0$; namely, the unique function satisfying
\medskip

\begin{equation}\label{dpf}
\left\{
\begin{array}{ll}
(\ell_0(x))^p=x, & \text{when $\delta(x)=0$} \\
\ell_0(x)=0, & \text{otherwise} 
\end{array}
\right.
\end{equation}
\medskip

The theory of differentially perfect fields (i.e., those $(K,\delta)$ such that $C_K=K^p$) is denoted by $\DPF_p^{\ell_0}$ and can be axiomatised by expanding $\DF_p$ by a sentence specifying \eqref{dpf} above. It follows that $\DCF_p^{\ell_0}$ is the model completion of $\DPF_{p}^{\ell_0}$. One could ask whether $\DPF_p^{\ell_0}$ is derivation-like with respect to $\SCF_{p,\infty}^\lambda$ (note that the underlying field of a $\DCF_p$ is a model of $\SCF_{p,\infty}$). We now prove this is not the case.

\begin{lemma}\label{notforSCF}
    The theory $\DPF_p^{\ell_0}$ is {\bf not} derivation-like with respect to ($\SCF_{p,\infty}^{\lambda},\ind^0$).
\end{lemma}
\begin{proof}
    Consider the function field $K=\mathbb F_p(t)$ with standard derivation $\delta=\frac{d}{dt}$. Note that $(K,\delta)\models \DPF_p$ since $C_K=\mathbb F_p$. Inside the model $\UU_0$ of $\SCF_{p,\infty}^{\lambda}$, find $s$ such that $t\ind^0_{\mathbb F_p}s$ and $\tp^{\SCF_{p,\infty}^{\lambda}}(t/\mathbb F_p)=\tp^{\SCF_{p,\infty}^{\lambda}}(s/\mathbb F_p)$. Equip $L=\mathbb F_p(s)$ with the derivation $\delta=\frac{d}{ds}$. We argue that there cannot be an $M$ as in condition (i) of derivation-like. It there were, $M$ would be a model of $\DPF_p^{\ell_0}$. In other words, $C_M=M^p$. Since $K\ind^0_{\mathbb F_p}L$, we obtain that $K$ and $L$ are $p$-disjoint over $\mathbb F_p$; and so $K\cdot L=\mathbb F_p(t,s)$ is an $\L_\lambda$-substructure of $M$. Hence, the extension $M/\mathbb F_p(t,s)$ is separable, which implies that $\mathbb F_p(t,s)$ is differentially perfect. But, since $\delta(t-s)=0$, this would imply that $t-s$ has a $p$-th root in $\mathbb F_p(t,s)$, which is impossible (as $t$ and $s$ are algebraically independent).
\end{proof}

One could further ask whether $\DPF_p^{\ell_0}$ is derivation-like with respect to $\ACF_p$. Again, this is not the case.

\begin{lemma}\label{notforACF}
The theory $\DPF_p^{\ell_0}$ is {\bf not} derivation-like with respect to ($\ACF_p,\ind^{\alg}$).
\end{lemma}
\begin{proof}
Consider the function field $K=\mathbb F_p(t)$ equipped with the standard derivation $\delta=\frac{d}{dt}$. Let $x$ be a differential indeterminate over $K$. Let $s:=t+x^p$. Then, the derivation on
$M:=K\langle x\rangle=K(x,\delta x,\dots)$
restricts to the standard derivation $\delta=\frac{d}{ds}$ on $L:=\mathbb F_p(s)$. Note that both $K$ and $L$ are differentially perfect and $K \ind_{\mathbb F_p}^{\alg}L$. However, the algebraic closure of the compositum $K\cdot L$ contains $x$ but it does not contain $\delta(x)$; namely, it is not a differential subfield. In other words, condition (ii) of derivation-like does not hold.
\end{proof}

\smallskip

\begin{remark} \
\begin{enumerate}
    \item [(i)] While $\DCF_p$ is a stable theory, the two proofs above show that forking independence does not have an obvious  algebraic description.\footnote{These examples grew out of discussions with Amador Martin-Pizarro.} Indeed,  the proof of Lemma~\ref{notforSCF} shows that $\ind^{0}=\ind^{\SCF_{p,\infty}}$ does not satisfy full existence in $\DCF_p$; while the proof of \ref{notforACF} shows that $\ind^{\alg}$ does not satisfy base monotonicity in $\DCF_p$. Currently the authors are not aware of an algebraic description of forking independence in this theory.

    \item [(ii)] We leave it as an exercise to check that $\DPF_p^{\ell_0}$ is almost derivation-like with respect to ($\ACF_p,\ind^{\alg}$) in the sense of Remark~\ref{almostder}, and hence Lemmas~\ref{diagram} and \ref{acl} apply to the theory $\DCF_p^{\ell_0}$.
\end{enumerate}
\end{remark}

\

\subsection{CCMs with meromorphic vector fields}

Our final examples demonstrates that our results apply beyond theories of fields. In this section we observe that the theory, recently formulated by Moosa \cite{Moo23}, of compact complex manifolds equipped with a ``differential" structure fits into our setup. 

\medskip

Recall that the theory CCM - compact complex manifolds - is the theory of the multi-sorted structure consisting of all compact complex manifolds (or rather all reduced and irreducible compact complex-analytic spaces) by naming as basic relations all closed complex-analytic subsets of finite cartesian products of sorts. See \cite{Moo05} for further details on this theory. However, in \cite{Moo23}, Moosa works in the seemingly more general setup of ``compactifiable" (rather than compact) complex-analytic spaces. Namely, he works in a definable expansion of CCM where there is a sort for each irreducible meromorphic variety. We denote by $\L_0$ the language of this expansion and continue to denote the theory of the expanded structure by CCM. The advantage of this expansion is that now sorts are closed under taking tangent bundles. We refer the reader to \cite[\S 2]{Moo23} for further details and explanations.

In the language $\L_\nabla=\L_0\cup\{\nabla_S:S\text{ is a sort of } \L_0\}$, where each $\nabla_S$ is a function symbol from sort $S$ to $TS$, Moosa considers the universal $\L_\nabla$-theory $\CCM_{\nabla}^\forall$ obtained by adding to $\CCM^\forall$ axioms specifying that $\nabla_S: S\to TS$ is a section to $\pi:TS \to S$ together with a compatibility condition of $\nabla$ with definable meromorphic maps between sorts (see \cite[Definition 3.3]{Moo23}). It turns out that, somewhat unintentionally, Moosa has proven that $\CCM_\nabla^\forall$ is derivation-like.

\begin{lemma}
    The theory $\CCM_\nabla^\forall$ is derivation-like with respect to $(\CCM,\ind^0)$ (here $\ind^0$ denotes forking independence).
\end{lemma}
\begin{proof}
In \cite[Lemma 6.2]{Moo23} Moosa proved a form of independent amalgamation that readily yields condition (i) of derivation-like. In addition, in \cite[Lemma 6.1]{Moo23}, he proves the uniqueness of differential CCM-structures of $\dcl$-closed sets inside $\acl^{\CCM}$-closures, yielding condition (ii) of derivation-like (or rather the weakening observed in Remark~\ref{aclA-and-useforCCM}(3)).
\end{proof}

In \cite[Theorem 5.5]{Moo23}, Moosa proves that the theory $\CCM_\nabla^\forall$ admits a model companion which he denotes by DCCM. Our results then yield some of the model-theoretic properties of DCCM deduced in \S6 and \S7 of \cite{Moo23}; e.g., completeness, quantifier elimination, description $\acl$ and $\operatorname{dcl}$, and stability.

\

\subsection{Theories with an automorphism}

For our final example, we exhibit some $\L_0$-theories $T_0$ where the $\L_0(\sigma)$-theory $T = T_0 \cup \{ \text{``$\sigma$ is an $\L_0$-automorphism''}\}$ is derivation-like with respect to $T_0$. We would like to thank the referee of this paper for suggesting this example.

Let $T_0$ be a stable $\L_0$-theory where $\dcl_0$ and $\acl_0$ coincide. Let $\ind^0$ be nonforking independence. Suppose we are given $A,B,C \models T^\forall$ with $C$ a common relatively $\acl_0$-closed $\L_0(\sigma)$-substructure and $A \ind^0_C B$. That  part (i) of the derivation-like axiom holds is precisely Lemma~6.3.18 of \cite{Wag00} -- this proof uses stability of $T_0$. For part (ii), we will need Remark~\ref{aclA-and-useforCCM}(3). Let $D = \dcl_0(D)$ be an $\L_0$-substructure of $M \models T$ such that
\[
\langle A, B \rangle_0 \leq_{\L_0} D \leq_{\L_0} \acl_0(A,B) \cap M.
\]
Then $D = \dcl_0(A,B)$. Now $d \in D$ is definable over $A$ and $B$ by some $\L_0$-formula $\phi(x, a, b)$. Then $\phi(x, \sigma(a), \sigma(b))$ defines some element $e$. Since $D$ is $\dcl_0$-closed, $e \in D$. We have that $\sigma|_M(d) = e$, and hence $D \leq_\L M$, and that any $\L_0$-automorphism of $D$ must send $d \mapsto e$, and hence this is the unique $\L_0$-automorphism on $D$ making it a model of $T^\forall$ and extending the ones on $A$ and $B$.

\begin{example}
\begin{enumerate}
    \item The theory of an infinite set with an automorphism is derivation-like with respect to the theory of an infinite set, and hence $\text{Infset} A$ is stable (see Section 5.7 in \cite{TentZiegler12}). Here, $A \ind^0_C B \iff A \cap B \subseteq C$.
    \item Let $K$ be a field and let $\L_0$ be the language of $K$-vector spaces. The theory of infinite $K$-vector spaces with an automorphism is derivation-like with respect to the theory of infinite $K$-vector spaces, and hence $\text{Vect}_K A$ is stable (see Section 5.7 in \cite{TentZiegler12}). Here $A \ind^0_C B \iff \dim(A/C) = \dim(A/BC)$.
\end{enumerate}
\end{example}


\end{document}